\documentclass{article}

\usepackage[a4paper,margin=25mm]{geometry}

\usepackage[T1]{fontenc}
\usepackage[utf8]{inputenc}
\usepackage{amsmath, amsthm, amsfonts, mathtools, amssymb, mathrsfs,multirow}
\usepackage{csquotes}
\usepackage[colorlinks, citecolor=red,linkcolor=blue]{hyperref}

\newcommand{\R}{\mathbb R}

\newcommand{\be}{\begin{equation}}
\newcommand{\ee}{\end{equation}}
\newcommand{\deriv}[2]{\frac{d #1}{d #2}}
\newcommand{\pderiv}[2]{\frac{\partial #1}{\partial #2}}

\newtheorem{thm}{Theorem}[section]

\newtheorem{lem}[thm]{Lemma}
\newtheorem{prop}[thm]{Proposition}

\newtheorem{rem}{Remark}

\newtheorem{conj}[thm]{Conjecture}
\numberwithin{equation}{section}

\newcommand{\userName}{Ifan Johnston \quad Vassili Kolokoltsov}

\usepackage{graphicx}
\graphicspath{{Pictures/}}

\allowdisplaybreaks

\begin{document}
\title{Mixed linear fractional boundary value problems}
\author{\userName}
\maketitle

\begin{abstract}
In this article we obtain two-sided estimates for the Greens function of fractional boundary value problems on $\R_+ \times \R_+ \times \R^d$ of the form 
\[(-\prescript{}{t_1}D^\beta_{0+*} - \prescript{}{t_2}D^\gamma_{0+*})u(t_1, t_2, x) = L_{x}u(t_1, t_2, x),\]
with some prescribed boundary functions on the boundaries $\{0\} \times \R_+ \times \R^d$ and $\R_+ \times\{0\}\times \R^d$. The operators $\prescript{}{t_1}D^\beta$ and $\prescript{}{t_1}D^\gamma$ are Caputo fractional derivatives of order $\beta, \gamma \in (0, 1)$ and $L_{x}$ is the generator of a diffusion semigroup: $L_x= \nabla \cdot(a(x) \nabla)$ for some nice function $a(x)$. The Greens function of such boundary value problems are decomposed into its components along each boundary, giving rise to a natural extension to the case involving $k \geq 2$ number of fractional derivatives on the left hand side.
\end{abstract}

\section{Introduction}
In the recent article \cite{johnston2019green}, we obtained two-sided estimates for the fundamental solution of fractional evolution equations of the form
\[-\prescript{}{t}D^\beta_{0+*} u(t,x) = L_x u(t,x),\]
where $\prescript{}{t}D^\beta_{0+*}$ is a Caputo-Dzherbashyan (CD) fractional derivative of order $\beta \in (0, 1)$ and $L_x$ is either the generator of a diffusion process, or a stable-like process (i.e, either a second order uniformly elliptic operator, or a spatially homogeneous pseudo-differential operator with variable coefficients). 
In this article we obtain two-sided estimates for the Greens function of the following boundary value problem,
\begin{align}\label{eqintro:bidimeqn}
-\prescript{}{t_1}D^\beta_{0+*} u(t_1, t_2, x) - \prescript{}{t_2}D_{0+*}^\gamma u(t_1, t_2, x) & = L_x u(t_1, t_2, x),\\
u(0, t_2, x) &= \phi_1(t_2, x),\nonumber\\
u(t_1, 0, x) &= \phi_2(t_1, x).\nonumber
\end{align}
See Section \ref{sec:mixedlinevol} for details on what is the object that we estimate. In Section \ref{sec:extndhigherdim} we look at higher dimensional version of (\ref{eqintro:bidimeqn}) in the sense that we have $k$ fractional derivatives on the left hand side, each acting on a different variable, 
\be\label{eqintro:kdimeqn}
-\sum_{i=1}^k \prescript{}{t_i}D^{\beta_{i}}_{0+*} u(t, x) = L_x u(t, x),\ee
where $(t, x)\in \R^k_+ \times \R^d$, with some specified boundary behaviour. 

The estimates obtained in this article can be used to study more general CD-type evolution equations (see \cite[Section 8.5]{kolokoltsov2019differential}) of the form
\be\label{eq:generalevol}
-\sum_{i=1}^k \prescript{}{t_i}D^{\nu_i(t_i, \cdot)}_{0+*} u(t, x) = L_x u(t, x),
\ee
where each $\nu_i(t_i, \cdot)$ is a L\'evy-type kernel, under the assumption that each $\nu_i(t_i, \cdot)$ has a density which is comparable to the density of a $\beta_i$-stable process. This was done for the case $k=1$ in \cite{johnston2019green}, so we do not repeat it here. Another natural (and essentially straightforward) extension of the estimates obtained in this article would be to replace $L_x$ with a homogeneous pseudo-differential operator with variable coefficients which generates a stable-like (Feller) process.

Boundary value problems such as (\ref{eqintro:bidimeqn}) and (\ref{eqintro:kdimeqn}) can be found in many areas of mathematics. A particularly noteworthy application can be found in the mathematics of insurance. Consider $k$ processes $(X_{t_1}^{\beta_1}(s), \cdots, X_{t_k}^{\beta_k}(s))$, where each $X^{\beta_i}_{t_i}(s)$ is generated by $-\prescript{}{t_i}D^{\beta_i}_{0+*}$. If each process corresponds to the wealth of a company, then whenever one of the coordinates hit zero, one of the companies have defaulted. Insurance companies are interested the \emph{ruin probability}, which is the probability of one of companies defaulting before a finite time horizon $T$. That is, if $\tau_0^{\beta_i}$ denotes the first time the process $X^{\beta_i}_{t_i}(s)$ hits zero,
\[\tau_0^{\beta_i} := \inf \{s > 0 : X_{t_i}^{\beta_i}(s) \leq 0\},\]
then the ruin probability is the quantity
\[\Psi(t_i, T) = \mathbb P[\tau_0^{\beta_i} < T].\]
See \cite{djehiche1993ruin, MR3520310, MR3515875, MR3007885,MR3343416,kumar2018fractional} for ruin probabilities of multidimensional risk models, or \cite{Asmussen2010ruin} for a broader treatment of ruin probabilities. Fractional version of compound Poisson processes are also of interest when looking at insurance risk processes, see \cite{leonenko2019limit}.
Similar kinds of questions also appear when looking at barrier options under one-dimensional Markov models, see \cite{MR3015232}. It is natural to consider multi-dimensional versions of these, \cite{MR3515877}, as investors often deal with basket options. A further natural appearance comes when considering portfolios of credit derivative obligations (CDO), which can be described by a Markov process in $\R^k_+$. Reaching a boundary of dimension $k - n$ means that $n$ out of $d$ bonds underlying the portfolio of CDOs have defaulted. It is natural in this setting to consider spatially non-homogeneous processes, since the behaviour of the processes should feel the approach to the boundary, which is not the case for L\'evy processes. This is then the setting of problem (\ref{eq:generalevol}). The series of articles \cite{scalas2000fractional, mainardi2000fractional, gorenflo2001fractional}, give a nice overview of the usage of fractional calculus and jump-diffusion processes in finance.

Another popular model these days is the so called Pearson diffusion, and also the fractional version, which are diffusion processes with polynomial diffusion coefficients, see \cite{leonenko2013correlation}. 
Fractional models are also finding new footing in theoretical physics, via fractional and non-local Schr\"odinger operators, see for example the articles \cite{kaleta2019zero, kaleta2018contractivity}.

Of more general interest in finance are \emph{affine processes} which live in $\R^k_+\times \R^d$, see \cite{duffie2003affine}.
Our final motivation for considering stable processes on $\R^2_+$ (i.e, (\ref{eqintro:bidimeqn}) without the spatial operator $L_x$), is the topic of limit order books. A simplified model would be that one coordinate of $X^{\beta_1, \beta_2}_{t_1, t_2}(s)$ is the volume of trades available at the best buy price while the other is the volume at the best sell price. The event that this process hits the boundary means that there are no more trades offered at that price and thus a price change occurs. These ideas will be developed in a forthcoming paper. See \cite{cont2012order, hambly2018limit} and references therein for related attempts at modelling order books using Brownian motions on the orthant and reflected SPDEs. 
\section{Preliminaries}
\subsection{Fractional derivatives and stable processes}
We begin by fixing some definitions and notations. 
For an open or closed convex subset $S$ of $\R^d$, $C(S)$ is the Banach space of continuous functions on $S$ equipped with the sup-norm, $C_\infty(S)$ is the closed subspace of $C(S)$ consisting of functions vanishing at infinity. $C^k(S)$ is a Banach space of $k$ times continuously differential functions with bounded derivatives on $S$ with the norm being the sum of the sup norms of the function itself and all its partial derivatives up to and including order $k$. 

For a subset $A \subset S$, let
\[C_{const A}(S) = \{f \in C(S)~:~\left.f\right|_A \text{ is a constant}\}.\]
The (left) Riemann-Liouville (RL) integral of order $\alpha >0$ is defined as
\[I^\alpha_af(t) = I^\alpha_{a+}f(t) = \frac 1{\Gamma(\beta)}\int_a^t (t-r)^{\beta-1}f(r)~\mathrm dr.\]
Then the fractional Caputo-Dzherbashyan (CD) of order $\beta \in (0, 1)$ is defined as
\[D^\beta_{a+*}f(t) = I^{1-\beta}_a\left[\deriv{}{t}f\right](t) = \frac 1{\Gamma(1-\beta)}\int_a^t (t-r)^{-\beta}\left[\deriv{}{r}f\right](r)~\mathrm dr, \quad t > a.\]
After some straightforward calculations, the CD derivative can be rewritten for smooth enough $f$ as
\be\label{eqdef:Cap}
D^\beta_{a+*}f(t) = \frac 1{\Gamma(-\beta)} \int_0^{t-a} \frac{f(t-r) - f(t)}{r^{1+\beta}}~\mathrm dr + \frac{f(t)-f(a)}{\Gamma(1-\beta)}{(t-a)^\beta}.
\ee
For $a = -\infty$ (and for smooth bounded integrable functions), the operator $D^{\beta}_{-\infty +*}$ is known as the fractional derivative in generator form,
\[
\deriv{^\beta}{x^\beta}f(x) := D^\beta_{-\infty + *} f(x) = \frac 1{\Gamma(-\beta)}\int_0^\infty \frac{f(x-z) - f(x)}{z^{1+\beta}}~\mathrm dz.
\]
The operators $-d^\beta f/dx^\beta$ with $\beta \in (0, 1)$ generate stable L\'evy subordinators (with inverted direction so that they are decreasing instead of increasing), see \cite[Chapter 3]{meerschaert2012stochastic}. Note that by (\ref{eqdef:Cap}), the CD derivative $D^\beta_{a+*}$ is obtained from $D^\beta_{-\infty +*}$ by the restriction of its action to the space $C_{const(-\infty, a]}(\R)$. Probabilistically, this means that for $\beta \in (0, 1)$ the process $\left(X^\beta_{x}(s)\right)_{s\geq 0}$ generated by $-D^\beta_{a+*}$ is a decreasing $\beta$-stable processes which is absorbed at $a$ on an attempt to cross it. We denote by $p^\alpha_s(t, r)$ the transition densities of the process $X^\alpha_t(s)$, and by $\tau_0^\alpha$ the time that $X_t^\alpha$ exits $(0, t]$, that is,
\[\tau_0^\alpha = \inf\{s > 0: X^\alpha_t(s) \leq 0\}.\]
Let $\mu_0^\alpha(s)$ be the density of the r.v $\tau_0^\alpha$.
\begin{rem}
The density $\mu_0^\alpha$ exists and is continuous due to classical results from the theory of stable processes.
\end{rem}
The density $\mu_0^\alpha(s)$ is given by,
\begin{align}\label{eq:dnstextraw}
\mu_0^\alpha(s) = \pderiv{}{s}\mathbb P\left[\tau_0^\alpha> s\right] &= \pderiv{}{s}\int_0^{t} p_s^\alpha (t, r)\mathrm dr\nonumber\\
& = \frac {t}\alpha s^{-1-\frac 1\alpha}w_\alpha (ts^{-1/\alpha}).
\end{align}
see \cite[Corollary 3.1]{meerschaert2004limit}, where $w_\alpha(z)$ is the density of an $\alpha$-stable random variable. We can also write this density conveniently as
\be\label{eq:dnstext}
t^\alpha \mu_0^\alpha(t^\alpha s) = \frac 1\alpha s^{-1-\frac 1\alpha} w_\alpha(s^{-\frac 1\alpha}).
\ee
The densities $w_\alpha(r)$ are one of the most important tools in studying the Greens function of evolution equations which involve fractional derivatives (of order at most 1). For $\alpha \in (0, 1)$ the density $w_\alpha(r)$ has the following asymptotic behaviour in a neighbourhood of $0$, \cite[Theorem~5.4.1]{zolotarev1999chance}
\be\label{eq:smlstbl}
w_\alpha(r) \sim C_{\alpha} r^{-\frac{2-\alpha}{2(1-\alpha)}}\exp\{-c_\alpha r^{-\frac{\alpha}{1-\alpha}}\}:= f_\alpha(r), \quad r \rightarrow 0,
\ee
and in a neighbourhood of $\infty$,
\[w_\alpha(r) \sim \tilde{C_\alpha} r^{-1-\alpha}, \quad r \rightarrow \infty.\]
\begin{rem}
One can show (\ref{eq:smlstbl}) by inverting the Laplace transform 
\[\exp\{-\lambda^\alpha\} = \mathbb E[ \exp\{-\lambda X^\alpha(1)\}] = \int_0^\infty e^{-\lambda x} w_\alpha(x)~\mathrm dx,\]
and applying the method known as the saddle point method.
\end{rem}
Due to the positivity of $w_\alpha(x)$, we can combine these behaviours so that there exists constants $C, \tilde{C} > 0$ such that
\be\label{eq:indistbl}
w_\alpha(r) \asymp C\mathbf 1_{\{r < 1\}} f_\alpha(r) +\tilde{C}\mathbf 1_{\{r \geq 1\}} r^{-1-\alpha}.
\ee
We will also be using the asymptotic behaviour of $\mu_\alpha^{0, t}(s)$, which follows from the asymptotic behaviour of $w_\alpha(s)$.
\begin{lem}\label{lem:extasmp}
For $\alpha \in (0, 1)$ the density $\mu_0^\alpha(s)$ of $\tau_0^\alpha$ has the following asymptotic behaviour at $0$ and $\infty$,
\[t^\alpha \mu_0^\alpha(t^\alpha s)\sim \left\{
\begin{array}{lc}
c_\alpha, & s \rightarrow 0,\\
c_\alpha s^{-1+ \frac 1{2(1-\alpha)}}\exp\{-c s^{\frac 1{1-\alpha}}\}, & s \rightarrow \infty.
\end{array}
\right.
\]
for some constants $c_\alpha > 0$.
\end{lem}
\begin{proof}
Since $w_\alpha(r)\sim r^{-1-\alpha}$ as $r\rightarrow \infty$, then $w_\alpha(r^{-\frac 1\alpha})\sim r^{1+\frac 1\alpha}$ as $r \rightarrow 0$. Thus using (\ref{eq:dnstext}), we have for $s \rightarrow 0$,
\[t^\alpha\mu_0^\alpha(t^\alpha s) = c_\alpha s^{-1-\frac 1\alpha} w_\alpha(s^{-\frac 1\alpha})\sim c_\alpha s^{-1- \frac 1\alpha}s^{1+\frac 1\alpha}= c_\alpha.\]
Using (\ref{eq:smlstbl}), note that $w_\alpha(r^{-\frac 1\alpha}) \sim f_\alpha(r^{-\frac 1\alpha})$ for $r \rightarrow \infty$. Thus for $s\rightarrow \infty$,
\begin{align*}
t^\alpha \mu_0^\alpha(t^\alpha s)= c_\alpha s^{-1-\frac 1 \alpha}w_\alpha(s^{-\frac 1\alpha})& \sim c_\alpha s^{-1-\frac 1\alpha}f_\alpha(s^{-\frac 1\alpha})\\
& = c_\alpha s^{-1 + \frac 1{2(1-\alpha)}} \exp\left\{-c_\alpha s^{\frac 1{1-\alpha}}\right\},
\end{align*}
as claimed.
\end{proof}
Let $Y_x(s)$ be a diffusion process with generator $L = \nabla\cdot(a(x)\nabla)$ for some symmetric measurable function $a$ on $ \R^d$. The estimates of Aronson, \cite{aronson1967bounds}, say that the transition densities $G^Y(s, x,y)$ of $Y_x(s)$ satisfy the following two-sided Gaussian estimates for all $s > 0$,
\be\label{eq:arnsn}
G^Y(s, x, y) \asymp s^{-\frac d2} \exp\left\{-c \frac{|x-y|^2}{s}\right\}.
\ee
Let $X^\alpha_r(s)$ be the process (independent of $Y_x(s)$) generated by $-D^{\alpha}_{0+*}$ (cf. \ref{eqdef:Cap}), which is a decreasing $\beta$-stable process absorbed at $0$ on an attempt to cross it. The transition density of the process $(Y_x(s), X^\alpha_r(s))$ is given by
\[G^{Y, \gamma}(s, r, x, y) = G^Y(s, x, y) s^{-\frac 1\gamma}w_\gamma(rs^{-\frac 1\gamma}).\]
The following result is obtained by applying Aronsons estimate for $G^Y$ and (\ref{eq:indistbl}) for $w_\gamma$.
\begin{lem}\label{lem:spcestm}
The transition density of $(X^\gamma_r(s), Y_x(s))$ satisfy the following estimates
\begin{itemize}
\item For $s \leq r^\gamma$,
\[G^{Y, \gamma}(s, r, x, y) \asymp C r^{-1-\gamma} s^{1-\frac d2}\exp\left\{ - c \frac{|x-y|^2}{s}\right\}.\]
\item For $s > r^\gamma$,
\[G^{Y, \gamma}(s, r, x, y) \asymp C r^{-\frac{2-\gamma}{2(1-\gamma)}} s^{\frac 1{2(1-\gamma)} -\frac d2}\exp\left\{- c\frac{|x-y|^2}{s} - c s^{\frac 1{1-\gamma}} r^{-\frac \gamma{1-\gamma}}\right\}\]
\end{itemize} 
\end{lem}
\subsection{Processes on the orthant}
Consider the process living on $\R^2_+$ defined by $X^{\beta, \gamma}_{t_1, t_2}(s):= (X_{t_1}^\beta(s), X_{t_2}^\gamma(s))$, where each coordinate a one-dimensional stable subordinator (with inverted sign) which absorbed at $0$, as described in the previous subsection. The process $X^{\beta, \gamma}_{t_1, t_2}(s)$ is generated by $-\prescript{}{t_1}D^\beta_{0+*}~-~\prescript{}{t_2}D_{0+*}^\gamma$, where $\beta, \gamma\in (0, 1)$, and it is started at $(t_1, t_2) \in \R_+\times \R_+$. 
For clarity, see Figure \ref{fig:ggplotorthant} for a typical sample path of $X^{\beta, \gamma}_{t_1, t_2}(s)$. We assume that the processes $X^\beta$ and $X^\gamma$ are independent. This independence assumption implies that the first time the process $X^{\beta, \gamma}_{t_1, t_2}$ hits the boundary of $\R_+\times \R_+$ is given by
\[\tau_0^{\beta, \gamma} = \min\left(\tau_0^\beta, \tau_0^\gamma\right).\]
\begin{figure}
\centering
\includegraphics[scale=0.3]{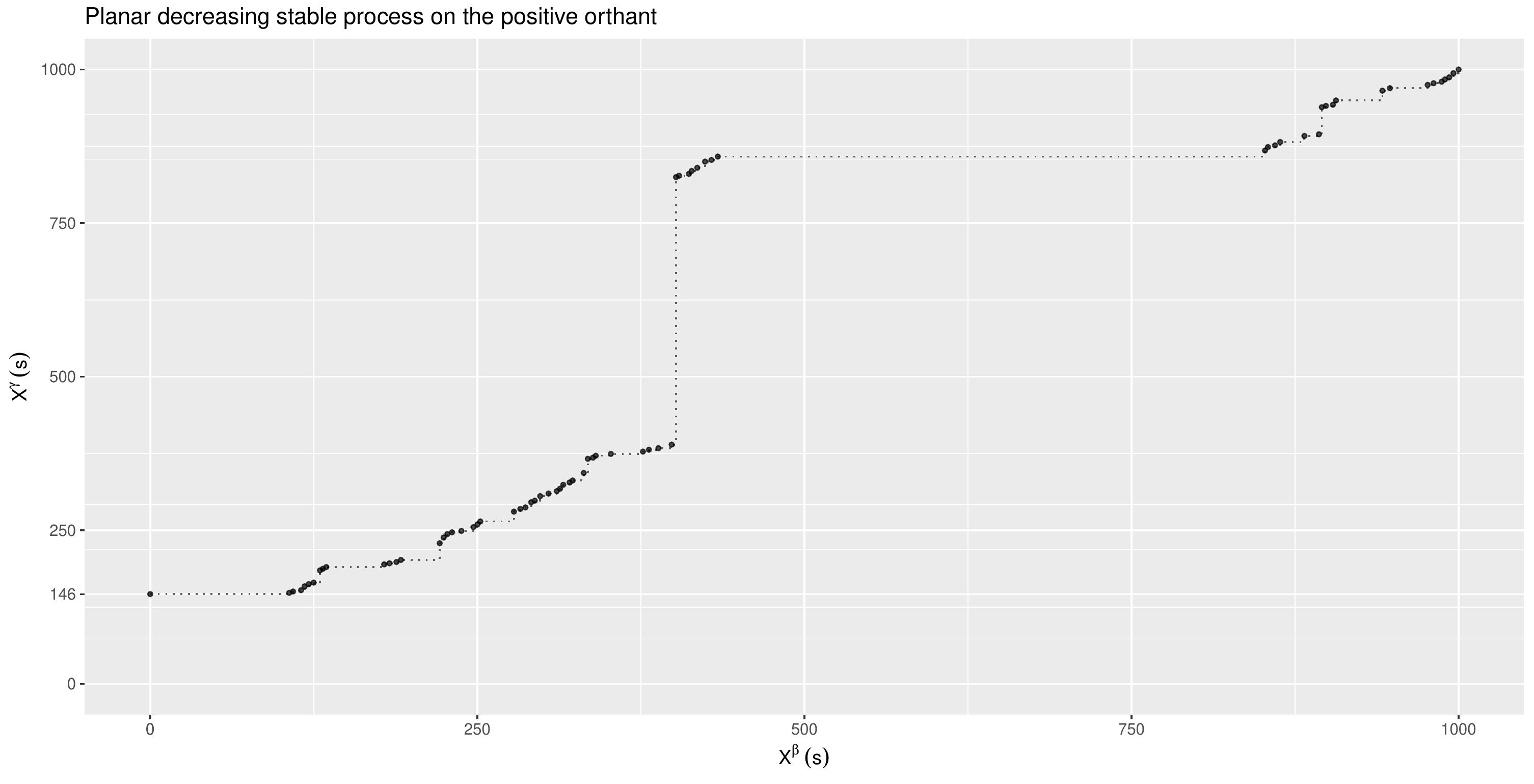}
\caption{Sample path of $X_{t_1, t_2}^{\beta, \gamma}(s)$ until the time $s=\tau_0^{\beta, \gamma}$ when it hits the boundary and $X^{\beta, \gamma}_{t_1, t_2}(\tau_0^{\beta, \gamma}) = (149, 0)$ in this case. Here $\beta = \gamma = 0.8$ and $t_1 = t_2 = 1000$. Made using the R packages ggplot2 \cite{ggplot2} and stabledist \cite{stabledist}.}
\label{fig:ggplotorthant}
\end{figure}

\section{Mixed linear evolution}\label{sec:mixedlinevol}
Consider the problem
\be\label{eq:mixlinbvp}
\begin{array}{rcl}
(-\prescript{}{t_1}D^\beta_{0+*} - \prescript{}{t_2}D^\gamma_{0+*})f(t_1, t_2, x) &=& Lf(t_1, t_2, x),\\
f(0, t_2, x)&=& \phi_1(t_2, x),\\
f(t_1, 0, x)&=&\phi_2(t_1, x).\\
\end{array}
\ee
Here $L$ is the generator of a Feller process $Y_x(s)$ started at $x\in \R^d$. For simplicity we take $L = \nabla\cdot(a(x)\nabla)$, where $a(x)$ is a symmetric, uniformly elliptic and measurable function. This means that $A$ generates a non-degenerate diffusion, with transition densities $G^Y(s, x, y)$ which satisfy Aronsons two-sided estimates (\ref{eq:arnsn}). 
\begin{rem}
Note that we could also obtain estimates for the Greens function in the case when $L$ is, say, a non-isotropic homogeneous pseudo-differential operator of order $\alpha \in (0, 2)$ whose symbol is of the form
\[\psi_\alpha(x, p) = |p|^\alpha w(x, p/|p|), \quad x\in \R^d,\]
where $w(x, \cdot)$ is some strictly positive function on $\mathbb S^{d-1}$. See \cite{eidelman2004analytic, kolokoltsov2000symmetric} for the relevant estimates for $G^Y$ in that case.
\end{rem}
\subsection{Well-posedness of the mixed boundary value problem}
Let us briefly discuss the well-posedness of problem (\ref{eq:mixlinbvp}). We only sketch the main steps, but see \cite[Chapter 8]{kolokoltsov2019differential}, \cite[Theorem 4.20]{hernandez2017generalised} or \cite[Section 4]{kolokoltsov2019mixed} for a full account of well-posedness for these types of problems.

For even more general operators $A$ generating Feller semigroups (and even generalised versions of Caputo-derivatives), one can obtain both uniqueness and the stochastic representation (\ref{eq:dynkinsolution}) of the solution to (\ref{eq:mixlinbvp}) via the Dynkin formula \cite[Theorem 5.1]{dynkin1965markov}.
To obtain existence of a classical solution, the main idea is to first transform the problem to an equivalent one involving zero boundary conditions and Riemann-Liouville fractional derivatives (by introducing a new unknown function $u(t_1, t_2, x) = f(t_1, t_2, x) - \mathbf 1_{\{t_2 > 0\}}\phi_1(t_2, x) - \mathbf 1_{\{t_1 > 0\}}\phi_2(t_1, x)$). This equivalent problem is then the following RL-type mixed boundary value problem,
\begin{align}\label{eq:mixlinbvpRL}
(-\prescript{}{t_1}D^\beta_{0+} - \prescript{}{t_2}D^\gamma_{0+} - A)u(t_1, t_2, x) =&~g_\phi(t_1, t_2, x),\\
u(0, t_2, x) = u(t_1, 0, x) = &~0,\nonumber
\end{align}
where
\[g_\phi(t_1, t_2, x) = (-\prescript{}{t_2}D^\gamma_{0+*} - A)\phi_1(t_2, x) + (-\prescript{}{t_1}D^\beta_{0+*} - A)\phi_2(t_1, x).\]
Notice that here we require $\phi_1$ and $\phi_2$ to be in the domain of the generators $(- \prescript{}{t_2}D^\gamma_{0+*} - A)$ and $(-\prescript{}{t_1}D^\beta_{0+*} - A)$ respectively. The unique solution in the domain of the generator to (\ref{eq:mixlinbvpRL}) is then found by applying the potential operator (of the semigroup $T_s^\beta T_s^\gamma e^{sA}$ generated by $(-\prescript{}{t_1}D^\beta_{0+} - \prescript{}{t_2}D^\gamma_{0+} - A)$) to the forcing term $g_\phi(t_1, t_2,x)$. The solution to the Caputo problem (\ref{eq:mixlinbvp}) is then recovered by undoing the shift by $\phi_1$ and $\phi_2$,
\begin{align*}f(t_1, t_2, x) = \mathbf 1_{\{t_1 = 0\}}\phi_1(t_2, x) &+ \mathbf 1_{\{t_2 = 0\}}\phi_2(t_1, x) \\
& + \mathbf 1_{\{t_1=0\}}\int_0^{t_2}\int_0^\infty e^{rA} G_{\beta, \gamma}(r, s_2)\mathrm dr(-\prescript{}{t_2}D_{0+*}^{\gamma} - A)\phi_1(t_2-s_2, x)\mathrm ds_2\\
& + \mathbf 1_{\{t_2=0\}}\int_0^{t_1}\int_0^\infty e^{rA} G_{\beta, \gamma}(r, s_1)\mathrm dr(-\prescript{}{t_1}D_{0+*}^{\beta} - A)\phi_2(t_1-s_1, x)\mathrm ds_1,
\end{align*}
where $G_{\beta, \gamma}(r, s)$ is the transition density of the process generated by $(-\prescript{}{t_1}D^{\beta}_{0+*} - \prescript{}{t_2}D^{\gamma}_{0+*})$. Rearranging and using \cite[Equation 4.126]{kolokoltsov2019mixed} we have
\[f(t_1, t_2, x) = E^{(\beta, \gamma)}\left[t_1^\beta L^\gamma\right]\phi_1(t_2, x) + E^{(\beta, \gamma)}\left[t_2^\gamma L^\beta\right]\phi_2(t_1, x)\]
for $L^\gamma:= (-\prescript{}{t_2}D_{0+*}^{\gamma} - A)$ and
$L^\beta := (-\prescript{}{t_1}D_{0+*}^{\beta} - A)$. Here $E^{(\beta, \gamma)}[D]$ are \emph{generalised operator-valued Mittag-Leffler functions}, which are introduced and extensively studied in the survey \cite{kolokoltsov2019mixed},
\[E^{(\beta, \gamma)}\left[t_1^\beta L^\gamma\right]\phi_1(t_2, x) = \int_0^\infty e^{s t_1^\beta L^\gamma}\phi_1(t_2, x)\mu_0^\beta(s)~\mathrm ds,\]
where $\mu_0^\beta(s)$ is the density of the exit time $\tau_0^\beta$ (cf. (\ref{eq:dnstextraw})).

\subsection{Estimates for Greens function}
As mentioned in the previous section, an application of the Dynkin formula followed by Doobs optimal stopping theorem gives the following stochastic representation of the solution (whenever it exists) to (\ref{eq:mixlinbvp}),
\be\label{eq:dynkinsolution}
f(t_1, t_2, x) = \mathbb E\left[\phi_1\left(X_{t_2}^\gamma(\tau_0^\beta), Y_x(\tau_0^\beta)\right)\mathbf 1_{\{\tau^\beta_0 < \tau_0^\gamma\}} + \phi_2\left(X_{t_1}^\beta(\tau_0^\gamma), Y_x(\tau_0^\gamma)\right)\mathbf 1_{\{\tau^\gamma_0 < \tau_0^\beta\}}\right]
\ee
A simple conditioning argument (see Appendix \ref{secapp:conditioning}), shows that this solution can be written as
\begin{multline*}
f(t_1, t_2, x) = \int_0^{t_2}\int_{\R^d} \phi_1(r, y)\left(\int_0^\infty G^Y(s, x, y)p^\gamma_s(t_2, r) \mu_0^\beta(s)~\mathrm ds\right)~\mathrm dy\mathrm dr\\
+\int_0^{t_1}\int_{\R^d} \phi_2(r, y)\left(\int_0^\infty G^Y(s, x, y)p^\beta_s(t_1, r) \mu_0^\gamma(s)~\mathrm ds\right)~\mathrm dy\mathrm dr,
\end{multline*}
Inserting (\ref{eq:dnstextraw}) for $\mu_0^\alpha$ and $\mu_0^\beta$,

\begin{align*}
f(t_1, t_2, x) = &~ \int_0^{t_2}\int_{\R^d} \phi_1(r, y)\left(\frac{t_1}\beta \int_0^\infty G^Y(s, x, y)s^{-1-\frac 1\beta - \frac 1\gamma} w_\gamma(r s^{-\frac 1\gamma})w_\beta(t_1 s^{-\frac 1\beta})~\mathrm ds\right)~\mathrm dy\mathrm dr\\
& +\int_0^{t_1}\int_{\R^d} \phi_2(r, y)\left(\frac{t_2}\gamma \int_0^\infty G^Y(s, x, y)s^{-1-\frac 1\beta - \frac 1\gamma} w_\beta(r s^{-\frac 1\beta})w_\gamma(t_2 s^{-\frac 1\gamma})~\mathrm ds\right)~\mathrm dy\mathrm dr\\
=&~\int_0^{t_2}\int_{\R^d}\phi_1(r, y) G_1^{(\beta, \gamma)}(t_1, r, x, y)~\mathrm dy\mathrm dr\\
& + \int_0^{t_1}\int_{\R^d}\phi_2(r, y) G_2^{(\beta, \gamma)}(t_2,r,x, y)~\mathrm dy\mathrm dr
\end{align*}
where $\mu_0^\beta(s)$ and $\mu_0^\gamma(s)$ are the densities of the exit times $\tau_0^\beta$ and $\tau_0^\gamma$, and
\[G^{(\beta, \gamma)}_1(t_1, r, x, y) := \frac{t_1}{\beta}\int_0^\infty G^Y(s, x, y) s^{-1-\frac 1\beta - \frac 1\gamma} w_\gamma(r s^{-\frac 1\gamma}) w_\beta(t_1 s^{-\frac 1\beta})~\mathrm ds,\]
and
\[G^{(\beta, \gamma)}_2(t_2, r, x, y) := \frac{t_2}{\gamma}\int_0^\infty G^Y(s, x, y) s^{-1-\frac 1\beta - \frac 1\gamma} w_\beta(r s^{-\frac 1\beta}) w_\gamma(t_2 s^{-\frac 1\gamma})~\mathrm ds.\]

Thus, the \emph{Greens function} associated to (\ref{eq:mixlinbvp}) are the coordinates of the integral kernel of the operator which acts on the boundary functions $\phi_1$ and $\phi_2$:
\begin{multline*}
f(t_1, t_2, x) = (\phi * G_{full})(t_1, t_2,x) = \int_{\partial \R^2_+\times \R^d}\phi(r_1, r_1, y) G^{\beta, \gamma}_{full}(t_1, t_2, r_1, r_2, x, y)~\mathrm dy\mathrm dr_1\mathrm dr_2\\
= \int_{\partial \R_+ \times \R^d}\phi_1(r_2, y)G^{(\beta, \gamma)}_1(t_1, r_2, x, y)\mathrm dr_2\mathrm dy + \int_{\partial \R_+ \times \R^d} \phi_2(r_1, y) G^{(\beta, \gamma)}_2(t_2, r_1, x, y)\mathrm dr_1\mathrm dy\\
= (\phi_1 * G_1^{(\beta, \gamma)})(t_1, x) + (\phi_2 * G_2^{(\beta, \gamma)})(t_2, x).
\end{multline*}
\begin{rem}
More generally, the function
\[f(x) = (\phi*G_{L})(x) = \int_{\partial X}\phi(z)G_{L}(x, z)~\mathrm dz,\]
solves the boundary value problem
\begin{align*}
Lf(x) =&~0, \quad x\in X,\\
f(z) =&~\phi(z), \quad z\in \partial X,
\end{align*}
where $\phi$ is a suitable function on the boundary of $X$.
\end{rem}
For this reason, to obtain global two-sided estimates for the \emph{full Greens function} $G_{full} = (G^{(\beta, \gamma)}_1, G^{(\beta, \gamma)}_2)$, it suffices to obtain estimates for $G_1^{(\beta, \gamma)}$, since the estimates for $G^{(\beta, \gamma)}_2$ will be the same up to exchanging coordinates. For the sake of readability we drop the subscripts from $G^{(\beta, \gamma)}_1$ and $t_1$ and look only at the function
\[G^{(\beta, \gamma)}(t, x;r,y) := G^{(\beta, \gamma)}_1(t_1, x;r,y).\]
Making the substitution $s = t^\beta z$, we have
\begin{align}\label{eq:grnspreestim}
G^{(\beta, \gamma)}(t, x;r,y) =&~ t^{-\frac{\beta}\gamma}\int_0^\infty G^Y(t^\beta z, x, y)z^{-1-\frac 1\beta - \frac 1\gamma} w_\gamma(r t^{-\frac \beta\gamma}z^{-\frac 1\gamma}) w_\beta(z^{-\frac 1\beta})~\mathrm dz\nonumber\\
=&~ \int_0^\infty G^{Y, \gamma}(t^\beta z,r, x, y) t^\beta\mu_0^\beta(t^\beta z)~\mathrm dz,
\end{align}
where $G^{Y, \gamma}$ and $\mu_0^\beta$ are as in Lemma \ref{lem:extasmp} and Lemma \ref{lem:spcestm}.
Let $\Omega:= |x-y|^2 t^{-\beta}$, $A = r^\gamma t^{-\beta}$.
\begin{prop}\label{prop:smallA} For $(t,r,x,y)\in (0, \infty)\times (0, t_2)\times \R^d \times \R^d$ and $t_2\in (0, \infty)$, the following estimates hold,
\begin{itemize}
\item For $\Omega \leq 1$,
\be\label{eqprop:gOmleq1}
G^{(\beta, \gamma)}(t,r,x, y) \asymp C  t^{-\frac \beta\gamma - \frac{d\beta}2}A^{-1-\gamma} \left\{
\begin{array}{lc}
C, & d = 1,2, 3,\\
(|\log\left(\Omega(\max\{A^{-1}, 1\}\right)| + 1),& d = 4,\\
\Omega^{2 - \frac{d}2}, & d \geq 5.
\end{array}
\right.
\ee
\item For $\Omega \geq 1$,
\be\label{eqprop:gOmgeq1}
G^{(\beta, \gamma)}(t,r,x, y)\asymp C t^{-\frac \beta\gamma - \frac{d\beta}2} \Omega^{N_1} A^{N_2} \exp\left\{-\left(\Omega(\max\{A^{-1},1\}\right)^{\frac 1{2-\min(\beta, \gamma)}}\right\},\ee
where
\begin{align*}
N_1 &= -\frac d2\left(\frac{1-\alpha}{2-\alpha}\right) + \frac{1-\alpha}{2(2-\alpha)(1-\tilde{\alpha})}\\
N_2 &= -\frac d2 \left(\frac 1{2-\alpha}\right) + \frac 1{2(2-\alpha)(1-\tilde{\alpha})} + \frac 1{2(1-\alpha)} - \frac{2-\gamma}{2\gamma(1-\gamma)}\\
\alpha & = \min(\beta, \gamma)\\
\tilde{\alpha} & = \max(\beta, \gamma).
\end{align*}
\end{itemize}
\end{prop}

\begin{proof} We sketch the main ideas of the proof here, see Appendix \ref{secapp:proofprp} for the full details. After applying Lemma \ref{lem:extasmp} and Lemma \ref{lem:spcestm} in (\ref{eq:grnspreestim}), we end up with 4 integrals which contribute to the estimate for $G^{(\beta, \gamma)}$. For $\Omega \leq 1$, the main contribution comes from the integral
\[I_1 = t^{-\frac\beta\gamma - \frac{d\beta}2}A^{-1-\frac 1\gamma} \int_0^{A\wedge 1} z^{1-\frac d2}\exp\left\{-\Omega z^{-1}\right\}~\mathrm dz.\]
After a substitution of $w = \Omega z^{-1}$, we immediately recognise the integral form of the incomplete gamma function, see Appendix \ref{secapp:IGF},
\[I_1 = t^{-\frac\beta\gamma - \frac{d\beta}2}A^{-1-\frac 1\gamma}\Omega^{2-\frac d2}\int_{(A^{-1}\vee 1)\Omega}^{\infty} w^{\frac d2 - 3}\exp\{-\Omega w\} ~\mathrm dw.\]
Thus we have the two-sided estimate for $I_1$ for $\Omega \leq 1$,
\[I_1 \asymp C t^{-\frac \beta\gamma - \frac{d\beta}2}A^{-1-\gamma} \left\{
\begin{array}{lc}
C, & d = 1,2, 3,\\
(|\log\left(\Omega(\max\{A^{-1}, 1\}\right)| + 1),& d = 4,\\
\Omega^{2 - \frac{d}2}, & d \geq 5.
\end{array}
\right.
\]
Since the integral $I_1$ is the main contributor to the estimate, this proves (\ref{eqprop:gOmleq1}).
For $\Omega \geq 1$, the main contribution to the estimate comes from the integral
\[I_4 = t^{-\frac \beta \gamma- \frac{d\beta}2}A^{-\frac{2-\gamma}{2\gamma(1-\gamma)}}\int_{A\vee 1}^\infty z^{-\frac d2-1+\frac 1{2(1-\beta)} + \frac 1{2(1-\gamma)}} \exp\left\{-\Omega z^{-1} - A^{-\frac 1{1-\gamma}} z^{\frac 1{1-\gamma}}-z^{\frac 1{1-\beta}}\right\}~\mathrm dz.\]
To estimate this integral, let $\alpha = \min(\beta, \gamma)$ and $\tilde{\alpha} = \max(\beta, \gamma)$. Then as an upper (resp. lower) bound for $I_4$ we replace the powers in the exponential term with $\alpha$ (resp. $\tilde{\alpha}$). That is, the upper estimate
\[I_4 \leq C_1 t^{-\frac \beta \gamma- \frac{d\beta}2}A^{-\frac{2-\gamma}{2\gamma(1-\gamma)}}\int_{A\vee 1}^\infty z^{-\frac d2-1+\frac 1{2(1-\beta)} + \frac 1{2(1-\gamma)}} \exp\left\{-\Omega z^{-1} - A^{-\frac 1{1-\alpha}} z^{\frac 1{1-\alpha}}\right\}~\mathrm dz,\]
and the lower estimate
\[I_4 \geq C_2 t^{-\frac \beta \gamma- \frac{d\beta}2}A^{-\frac{2-\gamma}{2\gamma(1-\gamma)}}\int_{A\vee 1}^\infty z^{-\frac d2-1+\frac 1{2(1-\beta)} + \frac 1{2(1-\gamma)}} \exp\left\{-\Omega z^{-1} - A^{-\frac 1{1-\tilde{\alpha}}} z^{\frac 1{1-\tilde{\alpha}}}\right\}~\mathrm dz.\]
Then an application of Proposition \ref{prop:asymptoticcomputation} from the Appendix proves (\ref{eqprop:gOmgeq1}), and we are done.
\end{proof}

\section{Extension to higher dimension}\label{sec:extndhigherdim}
Let us outline how to extend the previous sections to the case where we have more than two fractional derivatives. Let $\mathcal O$ be the orthant in $\R^k$ defined by
\[\mathcal O:= \{(t_1, \cdots, t_k) \in \R^k, t_i \geq 0, i \in \{1, \cdots, k\}\}.\]
Let $\mathcal O_{i, 0}$ denote the collection of vectors $t_{i,0}$ from $\mathcal O$ whose $i$-th coordinate is zero, 
\[\mathcal O_{i, 0} := \{t_{i,0} = (t_1, \cdots, t_{i-1}, 0, t_{i+1}, \cdots, t_k)\}.\]
Define $h_{i,0}(t)$ to be the projection of $\mathcal O_{i, 0}$ onto the subspace $\mathcal O_i\subset \R^{k-1}$ by removing the coordinate which is zero, that is, $h_{i, 0}(t):\mathcal O_{i, 0}\mapsto \mathcal O_i$
\[h_{i, 0}(t_{i, 0}) = (t_1, \cdots, t_{i-1}, t_{i+1}, \cdots, t_k).\]
We look at the equations on $\mathcal O \times \R^d$,
\begin{align}\label{eq:kdimeqn}
\left(-\sum_{i=1}^k \prescript{}{t_i}D^{\beta_{i}}_{0+*} - L\right)f(t, x) & = 0, &\text{ on } \mathcal O \times \R^d,\\
f(t_{i, 0}, x) & = \phi_i(h_{i, 0}(t_{i, 0}), x), & \text{ on } \mathcal O_{i, 0} \times \R^d,\nonumber
\end{align}
where each $\phi_i$ is a function on $\mathcal O_i \times \R^d$. 
\begin{rem}
In order to have continuity of the solution to the above boundary value problem, we would need to also impose additional boundary conditions in order to ensure that the solution coincides at the points where the boundary meets - i.e, at the origin. Without this additional assumption we only have a \emph{generalised} solution, which is enough for our purposes.
\end{rem}

As before, let $X_{t_i}^{\beta_i}(s)$ denote the process started at $t_i\in \R_+$ generated by $D^{\beta_i}_{0+*}$ where $\beta_i\in (0, 1)$, and let $\tau_0^{\beta_i}$ denote the exit time of this process from $(0, \infty)$,
\[\tau_0^{\beta_i} := \inf \{ s > 0: X^{\beta_i}_{t_i}(s) \leq 0\}.\]
Let $X^{\overline{\beta}}_{t}(s) = (X^{\beta_1}_{t_1}(s), \cdots, X^{\beta_k}_{t_k})$ be the process on $\mathcal O$ generated by 
\[-\prescript{}{t}D^{\overline{\beta}}_{0+*} := -\sum_{i=1}^k \prescript{}{t_i}D^{\beta_{i}}_{0+*},\]
and due to the independence of each process $X^{\beta_i}_{t_i}$, the exit time of $X^{\overline{\beta}}_{t}(s)$ from the orthant $\mathcal O$ is given by
\[\tau_0^{\overline{\beta}} = \min_{i\in \{1, \cdots, k\}}\tau_0^{\beta_i}.\]

For $t\in \R^k_+$, let $B_{i}(t)$ denote the subset of $\mathcal O_i$ defined by
\[B_i(t) := \{ r\in \mathcal O_i,~r_j \leq t_j,~j \neq i\},\]
i.e, $B_i$ consists of elements of the form
\[[0, t_1]\times \cdots \times [0, t_{i-1}]\times [0, t_{i+1}]\times \cdots \times [0, t_k]\in \mathcal O_i.\]
The solution to (\ref{eq:kdimeqn}) is given by
\begin{align*}
f(t, x) &= \mathbb E\left[\sum_{i = 1}^k \phi_i(h_{i, 0}(X^{\overline{\beta}}_{t}(\tau_0^{\overline{\beta}}), Y_x(\tau_0^{\overline{\beta}}))\mathbf 1_{\{\tau_0^{\overline{\beta}} = \tau_0^{\beta_i}\}}\right]\\
& = \sum_{i=1}^k \left( \int_{B_i(t)}\int_{\R^d} \phi_i(r,y))\left(\int_0^\infty p^Y(s, x, y) \prod_{j\neq i}^k p_s^{\beta_j}(t_j, r_j) \mu_0^{\beta_i}(s)~\mathrm ds\right)~\mathrm dy\mathrm dr\right)
\end{align*}
\begin{rem}
The same kind of conditioning argument works here, once the appropriate notation is adapted, see Appendix \ref{secapp:conditioning}.
\end{rem}
Thus the objects we are interested in is
\[G^{(\beta_i)}(t_i, x;r, y) = \int_0^\infty p^Y(s, x, y)\prod_{j\neq i}^k p^{\beta_j}_s(t_j, r_j)\mu_0^{\beta_i}(s)~\mathrm ds,\]
where $(t_i, x) \in \R_+ \times \R^d$, and $(r, y) \in \mathcal O_i \times \R^d$.
Note that
\[\prod_{j\neq i}^k p^{\beta_j}_s(t_j, r_j) = \prod_{j\neq i}^k s^{-\frac 1{\beta_j}}w_{\beta_j}(r_j s^{-\frac 1{\beta_j}})\]
and
\[\mu_0^{\beta_i}(s) = \frac{t_i}{\beta_i}s^{-1-\frac 1{\beta_i}}w_{\beta_i}(t_i s^{-\frac 1{\beta_i}}),\]
thus
\begin{align*}G^{(\beta_i)}(t_i, x;r, y) &= \frac{t_i}{\beta_i}\int_0^\infty p^Y(s, x, y) s^{-1-\sum_{i=1}^k \frac{1}{\beta_i}}\prod_{j\neq i}^k w_{\beta_j}(r_j s^{-\frac 1{\beta_j}})w_{\beta_i}(t_i s^{-\frac 1{\beta_i}})~\mathrm ds\\
& = \frac{t_i}{\beta_i}\int_0^\infty G^{Y}(s, r, x, y) \mu_0^\beta(s)~\mathrm ds
\end{align*}
Focusing on the first coordinate, we have
\begin{align*}G^{(\beta_1)}(t_1, x;r, y) &= \frac{t_1}{\beta_1}\int_0^\infty p^Y(s, x, y) s^{-1-\sum_{i=1}^k \frac{1}{\beta_i}}\prod_{j= 2}^k w_{\beta_j}(r_j s^{-\frac 1{\beta_j}})w_{\beta_1}(t_1 s^{-\frac 1{\beta_1}})~\mathrm ds\\
& = \frac{t_1}{\beta_1}\int_0^\infty G^{Y}(s, r, x, y) \mu_0^{\beta_1}(s)~\mathrm ds
\end{align*}
where $\mu_0^\beta(s)$ is the density of the exit time $\tau_0^{\beta_1}$, and $G^{Y, k}(s, r, x, y)$ is the density of the process 

 $(Y_x(s), X_{r_2}^{\beta_2}(s), X_{r_3}^{\beta_3}(s), \cdots, X^{\beta_k}_{r_k}(s))$. We assume that as usual $Y_x$ is a diffusion process, so that $p^Y$ satisfies Aronsons estimates, 
\begin{align}\label{eq:grnskterms}
G^{Y, k}(s, r, x, y) =&~p^Y(s, x, y)s^{-\sum_{i=2}^k \frac 1{\beta_i}}\prod_{j=2}^k w_{\beta_j}(r_j s^{-\frac 1{\beta_j}})\\
\asymp &~ s^{-\frac d2}\exp\left\{-\frac{|x-y|^2}s \right\} s^{-\sum_{i=2}^k \frac 1{\beta_i}}\prod_{j=2}^k w_{\beta_j}(r_j s^{-\frac 1{\beta_j}})\nonumber\\
\asymp & ~s^{-\frac d2} \exp \left\{-\frac{|x-y|^2}s\right\}\left(s^{k-1}\prod_{j=2}^k r_j^{-1-\beta_j}\mathbf 1_{\{s_j < r_j^{\beta_j}\}}\right.\nonumber\\
& + \cdots\nonumber\\
& +s^{n-1}\prod_{j=2}^n r_j^{-1-\beta_j} \mathbf 1_{\{s_j < r_j^{\beta_j}\}}\prod_{i=n+1}^k f_{\beta_j}(r_is^{-\frac 1{\beta_i}})\mathbf 1_{\{s_i > r_i^{\beta_i}\}}\nonumber\\
&+ \cdots \nonumber\\
& +\prod_{j=2}^k \left. s^{\frac 1{2(1-\beta_j)}}\exp\left\{- c_{\beta_j} \left(r^{-\beta_j}s\right)^{\frac 1{1-\beta_j}}\right\}r_j^{-\frac{2-\beta_j}{2(1-\beta_j)}}\mathbf 1_{\{s_j > r_j^{\beta_j}\}}\right)\nonumber
\end{align}
where the cross terms runs from $n = k - 1$ down to $n = 1$ in the above above and are the mixtures of long and short tails. Note that we use the convention that $\prod_{i=2}^1 = 1$.

Let $A_1 = t_1^{-\beta_1}\prod_{i=2}^k r_i^{\beta_i}$ and $\Omega = |x-y|^2 t^{-\beta_1}_1$.
\begin{conj} For $(t_1, r,x,y) \in (0, \infty)\times \mathcal O_1 \times \R^d \times \R^d$, we have the following two-sided estimates for the Greens function $G^{(\beta_1)}$,
\begin{itemize}
    \item For $\Omega \leq 1$,
\begin{align*}
G^{(\beta_1)}(t_1,r,x,y) \asymp  C t_1^{-\frac{d\beta_1}2}\Pi_1
\left\{
\begin{array}{lc}
C,& d \leq 2k-1,\\
\left|\log\left( \frac{\Omega}{\min\{A_1, 1\}}\right)\right| + 1,& d = 2k,\\
\Omega^{2-\frac d2},& d \geq 2k+1,
\end{array}
\right.
\end{align*}
where $\Pi_1 = \prod_{i=2}^k t_1^{-\frac{\beta_1}{\beta_i}} A_1^{-1-\frac 1{\beta_i}}$.
\item For $\Omega \geq 1$,
\[G^{(\beta_1)}(t_1,r,x,y) \asymp \Pi_2 t_1^{- \frac{d\beta_1}2}A_1^{N_1} \Omega^{N_1} \exp\left\{-\left(\frac{\Omega}{\min\{A_1,1\}}\right)^{\frac 1{2-\alpha}}\right\},\]
where $\Pi_2 = \left(\prod_{i=2}^k t_1^{-\frac{\beta_1}{\beta_i}}A_1^{-\frac{2-\beta_i}{2\beta_i(1-\beta_i)}}\right)$, $\alpha = \min\{\beta_1, \cdots, \beta_k\}$, and the powers $N_1$ and $N_2$ depend on $k$, $d$ and $\beta_i$ for $1 \leq i \leq k$.
\end{itemize}
\end{conj}
The calculations used to show the estimates in the case $k=2$ earlier in the article should work the same in this case, since the major contribution to the estimates should be the first term (respectively the last term) in (\ref{eq:grnskterms}) for small $\Omega$  (respectively large $\Omega$). We therefore omit the proof to avoid the cumbersome notation.
\newpage
\appendix
\section{Conditioning argument}\label{secapp:conditioning}
Recall that $\mu_0^\alpha$ is the density of the random variable $\tau_0^\alpha$ and $p^\alpha_s(t, r)$ are the transition densities of the monotone process $X^\alpha_t(s)$ started at $t\in (0, \infty)$.
\begin{prop} For $\beta, \gamma \in (0, 1)$, 
\be\label{eqprp:cndition}
\mathbb E\left[\phi_1(X_{t_2}^\gamma(\tau_0^\beta), Y_x(\tau_0^\beta))\mathbf 1_{\{\tau_0^\beta < \tau_0^\gamma\}}\right] = \int_0^{t_2}\int_{\R^d}\phi_1(r, y) \int_0^\infty p^Y(s, x, y) p^\gamma_s(t_2, r) \mu_0^\beta(s)~\mathrm ds\mathrm dy\mathrm dr,
\ee
and similarly,
\be\label{eqprp:c2}
\mathbb E\left[\phi_2(X_{t_1}^\beta(\tau_0^\gamma), Y_x(\tau_0^\gamma))\mathbf 1_{\{\tau_0^\gamma< \tau_0^\beta\}}\right] = \int_0^{t_1}\int_{\R^d}\phi_2(r, y) \int_0^\infty p^Y(s, x, y) p^\beta_s(t_1, r) \mu_0^\gamma(s)~\mathrm ds\mathrm dy\mathrm dr.
\ee
\end{prop}
\begin{proof} 
In the LHS of (\ref{eqprp:cndition}) condition first on $\{\tau_0^\beta = s\}$,
\begin{align*}\mathbb E\left[\phi_1(X_{t_2}^\gamma(\tau_0^\beta), Y_x(\tau_0^\beta))\mathbf 1_{\{\tau_0^\beta < \tau_0^\gamma\}}\right] &= \int_0^\infty \mathbb E \left[\phi_1(X_{t_2}^\gamma(s), Y_x(s))\mathbf 1_{\{s < \tau_0^\gamma\}}\right]\mu_0^\beta(s)~\mathrm ds.\\
\intertext{Due to the monotonicity of the process $X^\gamma_{t_2}$, the events $\{\tau_0^\gamma > s\}$ and $\{X^\gamma_{t_2}(s) > 0\}$ are equivalent. Thus we next condition on $\{X_{t_2}^\gamma(s) = r$\},}
&=\int_0^{\infty}\mathbb E\left[\phi_1(X_{t_2}^\gamma(s), Y_x(s))\mathbf 1_{\{X^\gamma_{t_2}(s) > 0\}}\right]\mu_0^\beta(s)~\mathrm ds\\
& = \int_0^{\infty} \int_0^{t_2}\mathbb E\left[\phi_1(r, Y_x(s))\mathbf 1_{\{r > 0\}}\right]\mu_0^\beta(s)p_s^\gamma(t_2, r)~\mathrm dr\mathrm ds.\\
\intertext{Finally conditioning on $\{Y_x(s) = y\}$ and rearranging, we have}
& = \int_0^{t_2} \int_{\R^d}\phi_1(r, y)\int_0^\infty p^Y(s, x, y) p_s^\gamma(t_2, r) \mu_0^\beta(s)~\mathrm ds\mathrm dy \mathrm dr
\end{align*}
where $p^Y(s, x, y)$ are the transition densities of the process $\left(Y_x(s)\right)_{s\geq 0}$ started at $x\in \R^d$. The proof of (\ref{eqprp:c2}) is similar and is omitted.
\end{proof}

\section{Proof of Proposition \ref{prop:smallA}}\label{secapp:proofprp}
Let $A := r^\gamma t^{-\beta}$, $\Omega := |x-y|^2 t^{-\beta}$. First we use Lemma \ref{lem:extasmp} to estimate the density $\mu_0^\beta$, then we use Lemma \ref{lem:spcestm} to estimate the spatial density
\begin{align*}
G^{(\beta, \gamma)}(t, x;r, y)=& ~\int_0^\infty G^{(Y, \gamma)}(t^\beta z, r, x, y)t^\beta \mu_0^\beta(t^\beta z)~\mathrm dz\\
\asymp &~\int_0^1 G^{(Y, \gamma)}(t^\beta z, r, x, y)~\mathrm dz + \int_1^\infty G^{(Y, \gamma)}(t^\beta z, r,x,y)z^{-1+ \frac 1{2(1-\beta)}}\exp\{-c_\beta z^{\frac 1{1-\beta}}\}~\mathrm dz\\
\asymp &~t^{-\frac\beta\gamma-\frac{d\beta}2}\int_0^1 z^{-\frac 1\gamma - \frac d2} \exp\left\{-\Omega z^{-1}\right\} w_\gamma(A^{\frac 1\gamma} z^{-\frac 1\gamma})~\mathrm dz\\
&+~ t^{-\frac \beta\gamma - \frac{d\beta}2} \int_1^\infty z^{-1+\frac 1{2(1-\beta)}-\frac 1\gamma - \frac d2} \exp\left\{-\Omega z^{-1}-c_\beta z^{\frac 1{1-\beta}} \right\} w_\gamma(A^{\frac 1\gamma} z^{-\frac 1\gamma})~\mathrm dz\\
\asymp &~ I_1 + I_2 + I_3 + I_4
\end{align*}
where
\begin{align*}
I_1 & = t^{-\frac \beta \gamma - \frac{d\beta}2} A^{-1-\frac 1\gamma} \int_0^{A\wedge 1} z^{1-\frac d2} \exp\left\{-\Omega z^{-1}\right\}~\mathrm dz~\mathbf 1_{\{A \in \R_+\}}\\
I_2 & = t^{-\frac \beta \gamma- \frac{d\beta}2} A^{-\frac{2-\gamma}{2\gamma(1-\gamma)}}\int_A^1 z^{\frac 1{2(1-\gamma)}-\frac d2}\exp\left\{-\Omega z^{-1} - A^{-\frac 1{1-\gamma}} z^{\frac 1{1-\gamma}}\right\}~\mathrm dz~\mathbf 1_{\{A < 1\}}\\
I_3 & = t^{-\frac \beta \gamma- \frac{d\beta}2} A^{-1-\frac 1\gamma}\int_1^A z^{\frac 1{2(1-\beta)}-\frac d2}\exp\left\{-\Omega z^{-1} - c_\beta z^{\frac 1{1-\beta}}\right\}~\mathrm dz~\mathbf 1_{\{A > 1\}}\\
I_4 & = t^{-\frac \beta \gamma- \frac{d\beta}2}A^{-\frac{2-\gamma}{2\gamma(1-\gamma)}}\int_{A\vee 1}^\infty z^{-\frac d2-1+\frac 1{2(1-\beta)} + \frac 1{2(1-\gamma)}} \exp\left\{-\Omega z^{-1} - A^{-\frac 1{1-\gamma}} z^{\frac 1{1-\gamma}}-z^{\frac 1{1-\beta}}\right\}~\mathrm dz~\mathbf 1_{\{A\in \R_+\}}
\end{align*}

Now we have 4 regimes to consider, which are
\begin{itemize}
    \item {\bf Case 1a):} $A \leq 1$ and $\Omega \leq 1$
    \item {\bf Case 1b):} $A \geq 1$ and $\Omega \leq 1$
    \item {\bf Case 2a):} $A \leq 1$ and $\Omega \geq 1$
    \item {\bf Case 2b):} $A \geq 1$ and $\Omega \geq 1$.
\end{itemize}
By directly comparing the powers of $z, \Omega$ and $A$ in the integrals above, we can reduce our attention to the integrals $I_1$ and $I_4$. Indeed for $\Omega \leq 1$ we have
\[0 = I_3 < I_4 \leq I_2 \leq I_1, \quad A \leq 1,\]
and
\[0 = I_2 < I_4 \leq I_3 \leq I_1, \quad A \geq 1.\]
For $\Omega \geq 1$ we have
\[0 = I_3 < I_1 \leq I_2 \leq I_4, \quad A \leq 1,\]
and
\[0 = I_2 < I_1 \leq I_3 \leq I_4, \quad A \geq 1.\]
Thus we have a preliminary two-sided estimate for $G^{(\beta, \gamma)}(t, r, x, y)$,
\[C_1 I_1 \leq G^{(\beta, \gamma)}(t, r, x, y) \leq C_2 I_1, \quad \text{ for } \Omega \leq 1,\]
and
\[C_3 I_4 \leq G{(\beta, \gamma)}(t, r, x, y) \leq C_4 I_4, \quad \text{ for } \Omega \geq 1,\]
for some constants $C_1, C_2, C_3, C_4$.
\subsection{Estimates for \texorpdfstring{$I_1$}{}}
For the first integral, we have for $A \leq 1$,
\[I_1 = t^{-\frac \beta\gamma} A^{-1-\frac 1\gamma} \int_0^{A} z^{1-\frac d2} \exp\left\{-\Omega z^{-1}\right\}~\mathrm dz~\]
Then for $\Omega \rightarrow 0$ and $A \rightarrow 0$,
\[I_1 \sim C_{\beta, d, \gamma} t^{-\frac\beta\gamma - \frac{d\beta}2} A^{-1-\frac 1\gamma}\left\{
\begin{array}{lc}
1,& d \leq 3,\\
|\log \Omega A^{-1}| + 1,& d = 4,\\
\Omega^{2-\frac d2}, & d \geq 5. 
\end{array}
\right.
\]
For $\Omega \rightarrow \infty$ and $A \rightarrow 0$,
\[I_1 \sim C_{\beta, d, \gamma} t^{-\frac\beta\gamma - \frac{d\beta}2}A^{1-\frac d2 -\frac 1\gamma}\Omega^{-1} \exp\left\{-\Omega A^{-1}\right\}.\]

For $A \geq 1$ we have
\[I_1 = t^{-\frac \beta\gamma  - \frac{d\beta}2} A^{-1-\frac 1\gamma} \int_0^1 z^{1-\frac d2}\exp\left\{-\Omega z^{-1}\right\}~\mathrm dz,\]
so for $\Omega\rightarrow 0$ and $A\rightarrow \infty$,
\[I_1 \sim C_{\beta, d, \gamma} t^{-\frac\beta\gamma - \frac{d\beta}2} A^{-1-\frac 1\gamma}\left\{
\begin{array}{lc}
1,& d \leq 3,\\
|\log \Omega| + 1,& d = 4,\\
\Omega^{2-\frac d2}, & d \geq 5. 
\end{array}
\right.
\]
For $\Omega\rightarrow \infty$ and $A\rightarrow \infty$,
\[I_1 \sim t^{-\frac \beta\gamma - \frac{d\beta}2}A^{-1-\frac 1\gamma}\Omega^{-1} \exp\left\{-\Omega\right\}.\]

\subsection{Estimates for \texorpdfstring{$I_4$}{}}
For $A \leq 1$,
\[I_4 = t^{-\frac\beta\gamma - \frac{d\beta}2}A^{-\frac{2-\gamma}{2\gamma(1-\gamma)}}\int_1^\infty z^n \exp\left\{-\Omega z^{-1} -A^{-\frac 1{1-\gamma}}z^{\frac 1{1-\gamma}} - c_\beta z^{\frac 1{1-\beta}}\right\}~\mathrm dz.\]
Let $\alpha:= \min(\beta, \gamma)$ and $\tilde{\alpha} = \max(\beta, \gamma)$. For bounded $\Omega\leq 1$, we have
\begin{align*}
I_4 \leq & t^{-\frac\beta\gamma - \frac{d\beta}2}A^{-\frac{2-\gamma}{2\gamma(1-\gamma)}}\int_1^\infty z^n \exp\left\{-c_\gamma A^{-\frac 1{1-\gamma}}z^{\frac 1{1-\gamma}} - c_\beta z^{\frac 1{1-\beta}}\right\}~\mathrm dz\\
\leq & t^{-\frac\beta\gamma - \frac{d\beta}2}A^{-\frac{2-\gamma}{2\gamma(1-\gamma)}}\int_1^\infty z^n \exp\left\{-c_\gamma A^{-\frac 1{1-\gamma}}z^{\frac 1{1-\alpha}}\right\}~\mathrm dz\\
\leq & t^{-\frac\beta\gamma -\frac{d\beta}2}A^{-\frac{2-\gamma}{2\gamma(1-\gamma)}}A^{\frac 1{1-\gamma}}\exp\left\{-C A^{-\frac 1{1-\gamma}}\right\},
\end{align*}
where we have used (\ref{eq:laplaceboundary}).
Next we use (\ref{prp:asmextended}) to get for $\Omega \geq 1$,
\begin{align*}
I_4&\leq  t^{-\frac \beta \gamma - \frac{d\beta}2} A^{-\frac{2-\gamma}{2\gamma(1-\gamma)}} \int_1^\infty z^n \exp\left\{-\Omega z^{-1} -A^{-\frac 1{1-\alpha}} z^{\frac 1{1-\alpha}}\right\}~\mathrm dz\\
&\sim C t^{-\frac\beta\gamma - \frac{d\beta}2} A^{-\frac{2-\gamma}{2\gamma(1-\gamma)}} \Omega^{\frac{2(n+1) - c}{2(c+1)}} A^{\frac{2c(n+1)+c}{2(c+1)}}\exp\left\{-C_2 \left(\Omega A^{-1}\right)^{\frac c{c+1}}\right\}, \quad \Omega A^{-1}\rightarrow \infty,
\end{align*}
where $c = \frac 1{1-\alpha}$ and $n = -\frac d2 - 1 + \frac 1{2(1-\beta)} + \frac 1{2(1-\gamma)}$.  Thus
\[I_4\leq C t^{-\frac \beta\gamma - \frac{d\beta}2} \left(\Omega A^{\frac 1{1-\alpha}}\right)^{-\frac{d}{2}\left(\frac{1-\alpha}{2-\alpha}\right)+\frac{1-\alpha}{2(2-\alpha)(1-\tilde{\alpha})}} A^{\frac 1{2(1-\alpha)}-\frac{2-\gamma}{2\gamma(1-\gamma)}}\exp\left\{-C\left(\Omega A^{-1}\right)^{\frac 1{2-\alpha}}\right\}.\]

Finally for $A \geq 1$, we have
\[I_4 = t^{-\frac\beta\gamma - \frac{d\beta}2} A^{-\frac{2-\gamma}{2\gamma(1-\gamma)}}\int_A^\infty z^n \exp\left\{-\Omega z^{-1} - c_\gamma A^{-\frac 1{1-\gamma}}z^{\frac 1{1-\gamma}} - c_\beta z^{\frac 1{1-\beta}}\right\}~\mathrm dz.\]
For bounded $\Omega$, but unbounded $A$ we have
\begin{align*}
I_4 \leq & t^{-\frac\beta\gamma - \frac{d\beta}2} A^{-\frac{2-\gamma}{2\gamma(1-\gamma)}}\int_A^\infty z^n \exp\left\{-c_\beta A^{-\frac 1{1-\gamma}}z^{\frac 1{1-\gamma}} - c_\gamma z^{\frac 1{1-\beta}}\right\}~\mathrm dz\\
\leq & t^{-\frac\beta\gamma - \frac{d\beta}2}A^{-\frac{2-\gamma}{2\gamma(1-\gamma)}}A^n \exp\left\{-c_\beta A^{\frac 1{1-\beta}}\right\}, \quad \text{ for } A\rightarrow \infty.
\end{align*}
For unbounded $\Omega$ and $A$, the term $A^{-\frac 1{1-\gamma}}$ is negligable since $A$ is large, then we apply the usual Laplace approximation Proposition \ref{prop:asymptoticcomputation} to get
\begin{align*}
I_4 \leq & t^{-\frac\beta\gamma - \frac{d\beta}2} A^{-\frac{2-\gamma}{2\gamma(1-\gamma)}}\int_A^\infty z^n \exp\left\{-\Omega z^{-1} - c_\beta A^{-\frac 1{1-\gamma}}z^{\frac 1{1-\gamma}} - c_\gamma z^{\frac 1{1-\beta}}\right\}~\mathrm dz\\
\leq & t^{-\frac\beta\gamma - \frac{d\beta}2}A^{-\frac{2-\gamma}{2\gamma(1-\gamma)}}\int_1^\infty z^n \exp\left\{-\Omega z^{-1} - c_\alpha z^{\frac 1{1-\alpha}}\right\}~\mathrm dz\\
\leq & t^{-\frac \beta \gamma - \frac{d\beta}2} A^{-\frac{2-\gamma}{2\gamma(1-\gamma)}}\int_0^1 w^{-n-2}\exp\left\{-\Omega w -c_\alpha w^{-\frac 1{1-\alpha}}\right\}~\mathrm dw\\
\leq & t^{-\frac \beta \gamma - \frac{d\beta}2} A^{-\frac{2-\gamma}{2\gamma(1-\gamma)}}\Omega^{\frac{n+1}{2-\alpha} - \frac{1}{2(2-\alpha)}}\exp\left\{-\Omega^{\frac 1{2-\alpha}}\right\}
\end{align*}
for $\Omega \rightarrow \infty$ and $A \geq 1$. For the lower bound of $I_4$, simply reverse the role of $\alpha$ and $\tilde{\alpha}$ in each case - otherwise structure of the estimates are the same.

\section{Asymptotic behaviour}
We describe here an important method used in asymptotic analysis, called the Laplace method. Our main references for asymptotic analysis are \cite{fedorjuk1987asymptotics, de1970asymptotic}. The goal of the Laplace method is to approximate integrals of the form
\[\int_a^b g(x) \exp\{-A h(x)\}~\mathrm dx,\]
as $A \rightarrow \infty$. As a motivating example, let $a = 1, b = \infty$, $h(x) = x$ and $g(x) = x^N$ for some integer $N > 0$\footnote{Of course this integral is just the upper-incomplete gamma function, whose asymptotic behaviour is well known.}. In this case, one could integrate by parts $N$ times, until the $x^N$ term dissapears, and one is left with a final integral
\[\int_1^\infty \exp\{-A x\}~\mathrm dx = A^{-1}\exp\{-A\},\]
so that, for sufficiently large $A$,
\[\int_1^\infty x^N \exp\{-A x\}~\mathrm dx = O(1)A^{-1} \exp\{-A\} + O(A^{-N-1}\exp\{-A\}).\]
Now the main idea is that the major contribution to the asymptotic behaviour of
\[\int_a^b g(x) \exp\{-A h(x)\}~\mathrm dx,\]
comes from a neighbourhood around the point at which the function $h(x)$ attains its miniumum value. Outside this neighbourhood, the contributions to the asymptotic behaviour are exponentially small. Our standard references for asymptotic methods are \cite{de1970asymptotic}, \cite{murray2012asymptotic} or \cite{fedorjuk1987asymptotics}. 
Let us assume that $h$ is a real continuous function, and that it attains its minimum at the boundary point $b$, that $h'(b)$ exists and $h'(b) > 0$. Moreover assume that $h(x) > h(b)$ (for $x > b$) and $h(x) \rightarrow \infty$ as $x\rightarrow \infty$. We will not recount the proof, but we state the asymptotic formula, 
\be\label{eq:laplaceboundary}\int_b^\infty g(x)\exp\{-A h(x)\}~\mathrm dx \sim g(b)(A h'(b))^{-1}\exp\{-Ah(b)\}, \quad A\rightarrow\infty.\ee
On the other hand, if the function $h$ has a minimum on the interior of the interval $(b, \infty)$, say at the point $\tilde{b}\in(b, \infty)$. Finally, assume that the derivative $h'(x)$ exists in some neighbourhood of $x = \tilde{b}$, that $h''(\tilde{b})$ exists and that $h''(\tilde{b}) > 0$. Then
\[
\int_b^\infty g(x)\exp\{-Ah(x)\}~\mathrm dx\sim g(\tilde{b})\sqrt{\frac{2\pi}{A h''(\tilde{b})}}\exp\{-Ah(\tilde{b})\}, \quad A\rightarrow\infty.
\]

\begin{prop}\label{prop:asymptoticcomputation}
Let $a> 1$, $N\in \R$, $c > 0$ and $\Omega \geq 1$. Then the following asymptotic formula holds as $\Omega \rightarrow \infty$,
	\[\int_0^1 w^N \exp\{-\Omega w - c w^{-a}\}~\mathrm dw\sim C_1(a, N, c) \Omega^{-\frac{2(N+1) + a}{2(a+1)}}\exp\left\{-C_2(c, a)\Omega^{\frac{a}{a+1}}\right\},\]
where $C_1(a, N, c) = a^{\frac{2(N+1) - 1}{2(a+1)}}c^{\frac{2(N+1) + 1}{2(a+1)}}\sqrt{\frac{2\pi}{a+1}}$, and $C_2(c, a) = c^{-\frac 1{a+1}}[ a^{\frac 1{a+1}} + a^{-\frac a{a+1}}]$.
\end{prop}

Further we have the slight extension to the above.

\begin{prop}\label{prp:asmextended}
Let $a, b > 1$, $n\in \R$, and $c := \min(a, b)$. Then as $\Omega A^{-1} \rightarrow \infty$,
\[\int_1^\infty z^n \exp\{-\Omega z^{-1} - A^{-a} z^a - z^b\}~\mathrm dz\sim C_1 \Omega^{\frac{2(n+1) - c}{2(c+1)}}A^{\frac{2c(n+1)+c}{2(c+1)}}\exp\left\{-C_2 \left(\Omega A^{-1}\right)^{\frac c{c+1}}\right\}.\]
\end{prop}

\subsection{Incomplete Gamma function}\label{secapp:IGF}
Here we give some formulas that we use in the main body of the article. The upper complete Gamma function, is defined by
\[\Gamma(s, A) = \int_A^\infty y^{s-1} \exp\{-y\}~\mathrm dy.\]
Equivalently after a change of variables $y = Aw$,
\[\Gamma(s, A) = A^{s}\int_1^\infty w^{s-1} \exp\{-Aw\}~\mathrm dw.\]
We have the following asymptotic behaviour of $\Gamma(s, A)$ for $A \rightarrow 0$,
\begin{align*}
\Gamma(s, A) \sim \left\{
\begin{array}{lc}
-s^{-1} A^{s},& s < 0,\\
(|\log A| + 1),& s = 0,\\
1 - s^{-s} A^s,& s > 0.
\end{array}
\right.
\end{align*}
Thus, for $A \leq 1$, 
\begin{align*}
A^{-s}\Gamma(s, A) \leq C_s\left\{
\begin{array}{lc}
1,& s < 0,\\
(|\log A|+ 1),& s = 0,\\
A^{-s}, & s > 0.
\end{array}
\right.
\end{align*}
For $A \rightarrow \infty$, we use the Laplace method (\ref{eq:laplaceboundary}) with $h(x) = x$, $b=1$, $g(x) = x^{s-1}$,
\[A^{-s}\Gamma(s, A) \sim A^{-1}\exp\{-A\}, \quad A\rightarrow \infty.\]

\bibliography{}
\bibliographystyle{alpha}

\end{document}